\documentclass[12pt,a4]{amsart}
\usepackage{amsmath}
\usepackage[top=1in, bottom=1in, left=1in, right=1in]{geometry}
\usepackage{amsthm,color}
\usepackage{amssymb}
\newcommand{\End}{\mathop{\mathrm{End}}\nolimits}

\newcommand{\el}{\mbox{$\mathcal{L}$}}

\newcommand{\ar}{\mbox{$\mathcal R$}}

\newcommand{\leqel}{\mbox{$\leq _{\mathcal L }$}}
\newcommand{\leqar}{\mbox{$\leq _{\mathcal R}$}}
\newcommand{\leqels}{\mbox{$\leq _{{\mathcal L}^{\ast}}$}}
\newcommand{\leqars}{\mbox{$\leq _{{\mathcal R}^{\ast}}$}}

\newcommand{\Aut}{\mathop{\mathrm{Aut}}\nolimits}
\newcommand\A{\Bbb A}
\newcommand\B{\Bbb B}
\newcommand\G{\Bbb G}

\newcommand\N{\Bbb N}

\newcommand{\en}{\operatorname{End}}
\newcommand{\pc}{\operatorname{PC}}

\newcommand{\Sing}{\operatorname{Sing}}

\newcommand{\ars}{\mbox{${\mathcal R}^{\ast}$}}

\newcommand{\els}{\mbox{${\mathcal L}^{\ast}$}}

\newcommand\T{\Bbb T}

\newcommand\s{^{\sharp}}
\newcommand\F{\Bbb {FG}}

\newcommand\im{\mbox{Im}\,}
\renewcommand\ker{\mbox{Ker}\,}
\newcommand\PC{\mbox{PC}}
\newcommand\R{\mathcal{R}}
\newcommand\gam{\gamma}
\newtheorem{lemma}{Lemma}[section]

\newtheorem{result}{Result}
\newtheorem{theorem}[lemma]{Theorem}

\newtheorem{prop}[lemma]{Proposition}
\newtheorem*{theorem*}{Theorem}
\theoremstyle{definition}
\newtheorem{defi}[lemma]{Definition}
\newtheorem{example}[lemma]{Example}

\newtheorem{question}[lemma]{Open Question}
\begin{document}

\title{Independence algebras,  basis algebras and the distributivity condition}
\thanks{}
\date{\today}
\author{Wolfram Bentz}
\address{School of Mathematics and Physical Sciences, University of Hull, Kingston-upon-Hull HU6 7RX, UK}
\email{w.bentz@hull.ac.uk}

\author{Victoria Gould}
\address{Department of Mathematics, University of York, Heslington, York YO10 5DD, UK}
\email{victoria.gould@york.ac.uk}

\begin{abstract} {\em Stable basis algebras} were
 introduced by Fountain and Gould and developed in a series of articles. 
 They form a class of universal algebras,  extending that of independence algebras, and reflecting  the way in which free modules over well-behaved domains generalise vector spaces. If a stable basis algebra 
 $\mathbb{B}$   satisfies the {\em distributivity condition}  (a condition satisfied by all the 
 previously known examples), it is a reduct of
 an  independence algebra $\mathbb{A}$. 
 Our first aim is to give an example of an independence algebra not satisfying the distributivity condition. 
 
Gould showed that if a stable basis algebra $\mathbb{B}$ with the distributivity condition has finite rank, then so does the independence algebra $\mathbb{A}$ of which it is a reduct, and
 in this case the endomorphism monoid $\End(\B)$ of $\B$ is a left order in the endomorphism monoid
 $\End(\A)$ of $\A$. We complete the picture by determining when $\End(\B)$ is a right, and hence a two-sided, order in  $\End(\A)$. In fact (for rank at least 2), this happens precisely when every element of 
 $\End(\A)$ can be written as $\alpha\s\beta$ where $\alpha,\beta\in\End(\B)$, 
 $\alpha\s$ is the inverse of $\alpha$ in a subgroup of $\End(\A)$
 and $\alpha$ and $\beta$ have the same kernel. This  is 
 equivalent to  $\End(\B)$ being  a special kind of left order in $\End(\A)$
 known as {\em straight}. 
 \end{abstract}


\maketitle
\section{Introduction and Preliminaries}\label{sec:intro}

The second author introduced the study of the endomorphism monoids of  universal algebras called {\em $v^*$-algebras}, which she named {\em independence algebras}. 
 These algebras appear first in an article of
Narkiewicz \cite{nar} and were inspired by Marczewski's study of notions of
independence, initiated in \cite{mar}
 (see \cite{gratzer} and the survey article \cite{urb}).  Such algebras may be defined via
properties of the closure operator $\langle -\rangle$ which takes a
subset of an algebra to the subalgebra it generates. In 
an independence algebra,
$\langle -\rangle$ must satisfy the {\em exchange property}, which
guarantees that we have a well behaved notion of {\em rank} for
subalgebras and hence for endomorphisms,
generalising that of the dimension of a vector space. Further,
independence algebras are {\em relatively free}. Precise definitions and further details
 may be found in
\cite{gould}.   We
remark that sets, vector spaces and free acts over any group are 
 examples of independence algebras. A full classification, which we will draw upon for this article, is given by Urbanik in 
 \cite{urb}.
 
We denote the monoid of endomorphisms of an algebra $\A$ 
 by $\End (\A)$. The study of $\End (\A)$  for an independence algebra $\A$  has flourished over the last twenty years \cite{ araujo,cameronszabo,fountainandlewin,fountainandlewin2,gray},  providing the framework for understanding the common behaviours of several fundamental examples of monoids, including full transformation monoids and the multiplicative monoids of matrix rings over division rings. For example, 
  if $\A$ is an independence algebra of finite rank $n$, then the
set $\Sing(\A)$ of endomorphisms of rank strictly less than $n$ forms an idempotent generated ideal \cite{fountainandlewin}. We remark that idempotent generated semigroups are ubiquitous, since every (finite) semigroup embeds into a (finite)  idempotent generated 
 semigroup \cite{Howie66}.  The study of idempotent generated semigroups has acquired significant momentum, due to recent advances (see, for example, \cite{gray:2012}) building upon Nambooripad's classical theory of free idempotent generated semigroups over biordered sets \cite{nambooripad:1979}.

The endomorphism monoid of an independence algebra $\A$ is regular.
Surprisingly, regularity of $\en (\A)$
is not necessary for the above results concerning
idempotent generation.
For example, the results of Laffey \cite{Laffey83} show that if $\A$
is a free module of finite rank $n$ over a Euclidean domain, 
then
 the set of non-identity
idempotents
of $\en (\A)$ generates the subsemigroup of endomorphisms of
rank strictly less than $n$.  Fountain and the second author introduced in \cite{basisi} a class of algebras
called {\em  stable basis algebras} that generalise free modules 
over Euclidean domains, in an attempt to put the results of Laffey,
 and later work of Fountain
\cite{fountainintegermatrices} and Ruitenberg \cite{wim},
 into a more general setting, an aim achieved in \cite{basisiii}.
Stable basis algebras are in particular relatively free algebras in which
the closure operator $\pc$ (pure closure)
satisfies the exchange property. Certainly independence algebras are
stable basis algebras. Finitely generated
free left modules over {\color{black} left Ore} Bezout domains 
and finitely generated free left $\mathbb{T}$-acts over any cancellative
 monoid $\mathbb{T}$ such that  finitely generated left ideals of $\mathbb{T}$ are principal,
 are examples of stable basis algebras. We recall that a 
{\em Bezout domain} is an integral domain (not necessarily commutative)
in which all finitely generated left and right ideals are principal; clearly any Euclidean domain is a (left) Ore Bezout domain.
As for independence algebras,
rank is well defined for 
subalgebras and endomorphisms of basis
algebras,
where now the rank is defined via the operator $\pc$\footnote{In earlier articles, rank in a basis algebra was referred to as {\em PC-rank} but, as there is no ambiguity, we simply use the term {\em rank} here.}. We give requisite definitions as we proceed through this article.

If $\A$ and $\B$ are algebras such that
the universe (that is, the underlying set)  $B$ of $\B$ is contained in the
universe of $A$ of $\A$, then $\B$ is a {\em reduct} of $\A$ if every basic
operation of $\B$ is the restriction to $B$ of a basic operation of $\A$.
Theorem 4.14 of \cite{gouldI} shows that if $\B$ is a stable basis algebra  satisfying the
distributivity condition, then $\B$ is a reduct of an independence algebra $\A$, having the same rank as $\B$. The
distributivity condition is stated precisely in  Section~\ref{sec:example}: essentially it says that unary operations distribute over basic $n$-ary operations, for $n\geq 2$. All previously mentioned 
examples of
basis algebras and independence algebras satisfy the distributivity condition. The main aim of Section~\ref{sec:example} is to prove:

\begin{result} Not all independence algebras satisfy the distributivity condition.
\end{result} 

 {\color{black} We achieve the above by giving a particular example of an $\mathbb{S}$-homogeneous independence algebra. Our argument is highly technical, for reasons that we explain in Section~\ref{sec:example}.}
 
Classical ring theory tells us that if $\mathbb{R}$ 
is a left Ore domain with division ring of quotients 
{\color{black} $\mathbb{D}$},  then for any $n\in\N$ the $n\times n$ matrix ring
$M_n(\mathbb{R})$ has ring of left quotients
$M_n(\mathbb{D})$, that is, $M_n(\mathbb{R})$ is a left order in $M_n(\mathbb{D})$ in the sense of ring theory. This means that every element of $M_n(\mathbb{D})$ can be written as
$U^{-1}V$ where $U,V\in M_n(\mathbb{R})$, and every cancellable element of $M_n(\mathbb{R})$ has an inverse in $M_n(\mathbb{D})$. The
endomorphism monoid of an arbitrary algebra, indeed of an arbitrary
independence algebra $\A$, need not be a ring.  We will therefore use the following notion 
of order,
due originally to Fountain and Petrich \cite{fp}; here $a\s$ denotes the inverse of $a$ in a (any) subgroup.

\begin{defi}\label{defi:order} Let $\mathbb{S}$ be a subsemigroup of a semigroup $\mathbb{Q}$. Then $\mathbb{S}$ is  a {\em left (right) order} in $\mathbb{Q}$ and $\mathbb{Q}$ is a
{\em semigroup of left (right) quotients} of $\mathbb{S}$ if every $q\in Q$ can be written as
$q=a\s b$ ($q=ba\s$) where $a,b\in S$, and every square-cancellable element of $S$ lies in a subgroup of $\mathbb{Q}$. We say that $\mathbb{S}$ is an {\em order} in $\mathbb{Q}$ and $\mathbb{Q}$ is a {\em semigroup of quotients of $\mathbb{S}$} if $\mathbb{S}$ is both a left and a right order in $\mathbb{Q}$.

In case  $\mathbb{S}$ is  a  left  order in $\mathbb{Q}$ and $a,b$ can be chosen above such that $a$ and $b$ generate the same principal right ideals in $\mathbb{Q}$, then we say that  $\mathbb{S}$ is 
{\em a straight left order} in  $\mathbb{Q}$, with corresponding definitions for {\em straight right order} and {\em straight order}.  
 \end{defi}
 
 We do not need here the precise definition of being square-cancellable, referring the reader to \cite{fountaingouldorders} -  it is a strong necessary condition for an element to lie in a subgroup of an oversemigroup. From 
 Theorems 3.4 and 3.11 of \cite{fountaingouldorders}, if $\mathbb{R}$ is a subring of a ring
 $\mathbb{Q}$ with identity, and provided  $\mathbb{Q}$ satisfies some weak conditions, certainly held by matrix rings over division rings, then $\mathbb{R}$ is a left order in  $\mathbb{Q}$ in the sense of ring theory if and only if it is a left order in the sense of Definition~\ref{defi:order}. Theorem 5.3 of \cite{gouldI} states that if
 $\B$ is  a finite rank stable basis algebra satisfying the distributivity condition and
 $\A$ is the independence algebra constructed of which it is a reduct, then
 $\en(\B)$ is a left order in $\en(\A)$. Moreover, $\en(\B)$ is a straight left order in
 $\en(\A)$ if and only if the monoid $\mathbb{T}$ of non-constant unary term operations of $\B$ acts by bijections on the constant subalgebra and (in the case rank
 $\B$ is at least 2) is both left and right Ore.   The second aim of this paper, achieved in Section~\ref{sec:right}, is to complete the picture by proving a series of results which give the following:
 
 \begin{result} Let   $\B$ be a stable basis algebra of finite rank $n\geq 2$ satisfying the distributivity condition and let $\A$ be the independence algebra from \cite{gouldI} of which it is a reduct, so that $\en(\B)$ is a left order in $\en(\A)$. Then 
 $\en(\B)$ is an order in $\en(\A)$ if and only if it is a straight left order
 in $\en(\A)$. 
 \end{result}

 We remark that, on the surface, being a straight left order has nothing to do with being a (two-sided) order. {\color{black} For example, any left Ore cancellative monoid $\mathbb{S}$ is a straight left order in a group $\mathbb{G}$, but will not be a right order unless $\mathbb{S}$ is also right Ore. }
 We make some remarks on particular examples, including free modules over domains, at the end of Section~\ref{sec:right}.
 
With $\B$ and $\A$ as in the above theorem, $\en(\B)$  sits inside $\en(\A)$ in a particularly nice way, known as being {\em fully stratified} (see Section~\ref{sec:right}).  In Section~\ref{sec:FS} we briefly address the question of an infinite rank stable basis algebra $\B$. Fountain and the second author showed, by directly checking a list of criteria from \cite{gould ab}, that the semigroup of endomorphisms $\en_f(\B)$ of finite rank was a fully stratified straight left order (without specifying the semigroup of left quotients). However,
 they were assuming that $\en(\B)$ is {\em idempotent connected}, which followed from the incorrect Proposition II 2.6 of \cite{mvl}.
 Without assuming  that  $\en(\B)$ is idempotent connected, we have a longer list of conditions to check. We restrict ourselves to a special case in order to do so, proving:
 
 \begin{result}  Let $\B$ be a stable basis algebra  that satisfies the distributivity condition, has no constants, and is such that  $\mathbb{T}$ is commutative. Then $\End_f(\B)$ is a fully stratified straight left order.
 \end{result} 
 
We conclude the paper in Section~\ref{sec:open} with a number of open questions.

 {\color{black} We presume a passing familiarity with the notation of Universal Algebra but, other than this,  we give all  necessary background for both Universal Algebra and Semigroup Theory as we proceed through the paper;} further specialist details may be found in \cite{basisi,basisii,basisiii,gould,gouldI,urb}.  For standard notions from semigroup theory, we refer the reader to \cite{cp} and  \cite{howie} and for universal algebra to \cite{mckenzie}.

\section{The distributivity condition for independence algebras}\label{sec:example}

If $\B$ is a stable basis algebra, then the monoid of non-constant unary term operations will be denoted by $\mathbb{T}$. It was shown in \cite{basisii} that $\mathbb{T}$
is {\em left Ore} (also known as {\em right reversible}), which means that for any $a,b\in T$ there exist $u,v\in T$ such that $ua=vb$. 

We quickly repeat some definitions from clone theory. An (abstract) {\em clone} on a set $A$ is a set of non-nullary operations on $A$ that contains all projections and is closed under
composition. Given a set $W$ of non-nullary operations on $A$, the clone {\em generated} by $W$ is the smallest clone containing $W$. The {\em clone of an algebra} $\A$ with underlying set $A$ is the clone generated by all non-nullary basic operations of $\A$. An (abstract or algebraic) {\em extended clone} is defined correspondingly by removing the restriction on nullary operations. 

\begin{defi}\label{def:dist} A basis algebra $\B$ satisfies the {\em distributivity condition} if the clone of $\B$ contains a generating set $W$ of operations such that for all $a\in T$ and $n$-ary operations $t \in W$, where $n\geq 2$, we have
\[a(t(x_1,\hdots, x_n))=t(a(x_1),\hdots, a(x_n)).\]
\end{defi}

Note that Definition~\ref{def:dist} is stated here more precisely than in \cite{gouldI}, since to show $\B$ does {\em not} satisfy the distributivity condition, we wish to show it is impossible to choose {\em any} generating set for the clone that witnesses this. 
Indeed, we see in Subsection~\ref{subs:lin} an example where the most natural choice of generating set does not witness the condition in Definition~\ref{def:dist}, but we can easily find another that does.

 Certainly free modules over rings, and acts over monoids,
hence our canonical examples of stable basis algebras, 
satisfy this condition. We have observed that all independence
algebras are stable basis algebras. Independence algebras are essentially $v^*$-algebras, which were completely determined (up to clone equivalence) in the
1960's; we refer the reader to the survey article of Urbanik \cite{urb}
for the details. 

In the beginning of Section 4 of \cite{gouldI}, it is claimed all independence algebras can be shown to satisfy the distributivity condition, with the possible exceptions of the 
 $\mathbb{S}$-{\em homogeneous} algebras or
 $\mathbb{Q}$-{\em homogeneous} algebras, where  $\mathbb{S}$ is a monoid and
 $\mathbb{Q}$ a quasifield. In the following, we will prove this claim, and address the remaining two types of independence algebras. In particular, we show that the distributivity condition  does not necessarily  hold 
 for $\mathbb{S}$-homogeneous independence algebras, thereby establishing {\bf Result 1}.

 The following subsections  consider the various classification types from \cite{urb}, giving full definitions of the various types. We will first address a technical difference between independence algebras and $v^*$-algebras, and how it affects this classification. The reader might want to note that the next paragraphs deal only with this technicality, and can be skipped without affecting the understanding of our  results.
 
 The issue at hand originates in the way that $\langle \emptyset \rangle$, the subuniverse generated by the empty set, is defined. In \cite{urb}, this is the set of all elements that are  images of  constant clone operations, while in the now prevailing definition, this is the subuniverse generated by the images of nullary operations (in both cases, if the algebra has no operations of the relevant type, we set $\langle \emptyset \rangle=\emptyset$). 
 
 A consequence of the definition from \cite{urb} is that nullary operation can effectively be ignored for classification purposes. Hence the results in \cite{urb} amount to a classification up to the clone of the algebra. If an (abstract) clone $C$ on a set has constant functions,  then there are algebras with and without nullary operations whose (algebraic) 
  clone is $C$, and their status as a $v^*$-algebra is not be effected by this difference. 
%
%
 
 For an independence algebra $\A$ with at least two elements one can show that $\langle \emptyset \rangle$ is exactly the subalgebra $[\emptyset]$ consisting of the images on non-nullary constant clone operations, so that $\A$ is a $v^*$-algebra. On the other hand, for a  (non-trivial) $v^*$-algebra $\B$ to be an independence algebra the extended clone needs to includes nullary operations for the elements of $[\emptyset]$.
The classification of \cite{urb} will be used with this understanding.
For simplicity, we will not explicitly list any additional nullary operations below.
 
 \subsection{All independence algebras of rank $0$ have the distributivity property}
 Here all elements are images of constant clone functions. It follows that $\T$ only contains the identity. The distribution property now follows trivially.
 
 \subsection{All linear independence algebras have the distributivity property}\label{subs:lin}
 Let $D$ be a division ring, $A$ the underlying set of a linear space over $D$, and $A_0\subseteq A$ the underlying set of a subspace. The linear independence algebra $\A$ given by 
 $(D,A,A_0)$ has underlying set $A$ and its basic operations are all operations of the form
 \[f(x_1,\dots,x_n)=\sum_{i=1}^{n} \, \lambda_i x_i +a,\]
 for each $1\le n$, $\lambda_i \in D, a\in A_0$. Clearly, every clone function of $\A$ is basic.
 
 For $\lambda \in D, a \in A_0$, let $f_{\lambda, a}$ be given by $f_{\lambda, a}(x)=\lambda x +a$. Clearly, all $f_{\lambda, a}$ are unary operations in $\A$, and in fact the monoid $\T$ contains exactly the functions $f_{\lambda, a}$ with $\lambda \ne 0$. 
 
 Moreover let $g(x_1,x_2,x_3)=x_1-x_2+x_3$, which is also a basic operation of $\A$. We claim that $W=\{g, f_{\lambda,a}: \lambda \in D, a \in A_0\}$ witnesses the  distribution condition. 
Note that 
$x+y=g(x,f_{0,0}(x),y)$. As  $+$ and the $f_{\lambda,a}$ clearly generate the clone of $\A$, so does $W$. 

Finally, for all $\lambda \ne 0$ and $a\in A_0$,
$$f_{\lambda,a}(g(x_1,x_2,x_3))=\lambda x_1 -\lambda x_2 +\lambda x_3 +a$$ 
$$= \lambda x_1+a -\lambda x_2-a +\lambda x_3 +a =g(f_{\lambda,a}(x_1),f_{\lambda,a}(x_2),f_{\lambda,a}(x_3)),$$
and the distributivity condition holds.

\medskip
We remark that, if $A_0 \ne \{0\}$, then a (natural) set of generators involving $+$ does not witness the distributivity condition. For  if $a \in A_0 \setminus\{0\}$, then
$$f_{1,a}(+(0,0))=a \ne a+a =  +(f_{1,a}(0),f_{1,a}(0)).$$

   \subsection{All affine independence algebras have the distributivity property}
 Let $D$ be a division ring, $A$ the underlying set of a linear space over $D$, and $A_0\subseteq A$ the underlying set of a subspace. The affine independence algebra $\A$ given by 
 $(D,A,A_0)$ has underlying set $A$ and its basic operations are all operations of the form
 $$f(x_1,\dots,x_n)=\sum_{i=1}^{n} \, \lambda_i x_i +a,$$
 for each $1\le n$, $a \in A_0, \lambda_i \in D$, such that $\Sigma_i \lambda_i=1$. Clearly, every clone function of $\A$ is basic, and $\T$ consists of the functions $f_b$ with
 $f_b(x)=x+b$, for all $b \in A_0$.
 
We can use the set of all clone functions as a generating set $W$. It witnesses the distributivity property, as, for  all $i \le 2$, $a, b \in A_0$, $\lambda_i \in D$,
$$\Sigma_i \lambda_i f_b(x_i) +a=\Sigma_i \lambda_i x_i +\Sigma_i \lambda_i b +a=\Sigma_i \lambda_i x_i + b +a = f_b(\Sigma_i \lambda_i+a).$$

 \subsection{The exceptional independence algebra has the distributivity property}
 The exceptional independence algebra $\A$ can be described as the algebra on a $4$-element set $A$ with a unary operation $i$ and a ternary operation $q$. Here $i$ is a product of two disjoint transpositions, and $q(x_1,x_2,x_3)$ is either the unique element not among its arguments (if all are different), or the argument that appears at least twice. It is straightforward to check that $T=\{i,1_A\}$, where $1_A$ is the identity map on $A$, and that $W=\{i,q\}$ witnesses the distributivity property. 
 
   \subsection{All group action independence algebras have the distributivity property}
 Let $A_0\subseteq A$ be sets, and $G$ a group acting on $A$, such that, for all non-identity $g \in G$, the fixed points of $g$ all lie in $A_0$, and $g(A_0) \subseteq A_0$.
 
 The group action independence algebra $\A$ corresponding to $(A,A_0,g)$ has underlying set $A$ and operations 
 $$ f_{g,n,i}(x_1,\dots,x_n)= g(x_i), \quad\quad f_{a,n}(x_1,\dots,x_n)=a,$$
for all $1\le n, 1\le i\le n, g \in G, a \in A_0$.  

Clearly, all clone operations are essentially unary. It follows that we may generate the clone with unary operations alone. Hence the distributivity condition holds trivially.

\subsection{All  $\mathbb{Q}$-{homogeneous}  independence algebras have the distributivity property}

A {\em quasifield} $\mathbb{Q}$ is a set $Q$ with at least two elements, together with two binary operations denoted by 
juxtaposition, and $-$, such that the multiplicative operation has a zero $0$,
the non-zero elements form a group under multiplication, and four further axioms hold (see 4.3 of \cite{urb}). 
In a $\mathbb{Q}$-{homogeneous}  independence algebra $\A$ over a quasifield $\mathbb{Q}$, all basic $k$-ary operations $f$ satisfy
$$f(a-ba_1,\dots, a-ba_k)=a-bf(a_1,\dots,a_k),$$
for all $a,b_1,\dots, b_k\in Q$,
where subtraction and multiplication are the operations from $\mathbb{Q}$. Setting $b=0$
and $a_i=a$ for $1\leq i\leq k$, we obtain that 
$f(a,\dots,a)=a$. 
An inductive argument gives that the identity is the only unary clone operation of $\A$. The distributivity property now follows trivially.

\subsection{Not all $\mathbb{S}$-homogeneous independence algebras have the distributivity condition}

Let $\mathbb{S}$ be a monoid such that all the non-invertible elements are left zeros. A good example is a group, which is exactly what we will take below. An $n$-ary operation on
$S$ is said to be $\mathbb{S}$-{\em homogeneous} if for all $s,s_1,\hdots, s_n\in S$ we have
\[f(s_1,\hdots, s_n)s=f(s_1s,\hdots, s_ns).\]
Since $\mathbb{S}$ is a monoid, the operations $f(x)=sx$ ($s\in S$) are the only
$1$-homogeneous operations. A $\mathbb{S}$-homogeneous independence algebra has underlying set $S$ and the basic operations form a set $V$ of $n$-homogeneous operations on
$S$ containing all the $1$-homogeneous operations. The aim of this subsection is to show that with a careful choice of $\mathbb{S}$ and $V$, the resulting independence algebra
$\A$ does not satisfy the distributivity condition.

Let $Z=\{z_1,z_2,\dots\}$ be a countably infinite set
and let $E=\{z_2,z_4,\dots\}$.  Let $\F=\F(Z)$ be the free group  over $Z$ with 
identity $1$ and underlying
set  $FG(Z)$, which we denote for brevity by $F$. In the following,
concatenation will always refer to the group operation of $\F$. 

Let $F_+\subseteq F$ be the set of all non-identity elements of $F$ whose normal form does not include any negative exponents, so that $F_+$ is the underlying set of the copy of the free semigroup $\mathbb{FS}$ on $Z$ sitting inside $\F$. 

Since $F$ and $E$ are both countably infinite we may choose a function $h:F \to E$, where 
$w\mapsto h_w$, satisfying the following conditions:
\[\begin{array}{lrclllll}
(h1)&h_{z_1z_2^{-1}}&=&z_6&(h2)&h_{z_3z_2^{-1}}=z_8&&\\
(h3)& h_{z_1z_4^{-1}}&=&z_{10}&(h4)&h\mbox{ is injective.}&&\end{array}\]

Now let $\A=\langle F; \{\nu_c^\A\}_{c \in F}, g^\A\rangle$ where:
\begin{enumerate}
  \item for each $c \in F$, $\nu_c^\A$ is the unary operation given by $\nu_c^\A(w)=cw$ (i.e. $\nu_c^\A$ acts as  left translation by the element $c$ in the group $\F$);
  \item $g^\A$ is the binary operation given by $g^\A(w_1,w_2)=h_{w_1w_2^{-1}}w_2$.
\end{enumerate}

\begin{lemma} The algebra $\A$ is a monoid independence algebra with underlying monoid
$\F$.
\end{lemma}
\begin{proof} We first remark that as $\F$ is a group, it has no non-invertible elements, and so is a suitable monoid from which to build a monoid independence algebra. 
We have remarked that $F$-homogeneous unary operations are left translations
and by construction, all left  translations $\nu_c^\A, c\in F$ are  basic in 
$\A$.  

Finally, for all $w_1,w_2,w' \in F$, we have $$g^\A(w_1w',w_2w')=h_{w_1w'(w_2w')^{-1}}w_2w'=h_{w_1w_2^{-1}}w_2w'=g^\A(w_1,w_2)w',$$
so that $g^\A$ is $F$-homogeneous and hence $\A$ is a monoid independence algebra.
\end{proof}

In order to show that $\A$ does not have the distributivity property, we need to examine the clone of $\A$. We let $L=\{\{\nu_c\}_{c \in F}, g\}$, be the language of $\A$,
$X=\{x_1,x_2,\dots,\}$ a countably infinite set of variables (which we may think of as being linearly ordered according to their subscripts), and $T$ the set of terms in the language $L$ over $X$.  The elements of the clone are obtained from 
the  interpretation in $\A$ of elements in $T$ (and their compositions with projections).

For $i,j \in \mathbb{N}$ with $i\le j$ let $\pi^j_i:F^j\to F^i$ be given by $(w_1,\dots,w_j)\mapsto  (w_1,\dots,w_i)$, i.e. $\pi^j_i$ is the projection to the first $i$ coordinates.

For each term $t \in T$, let $a(t)$ be the largest $n$ such that the variable  $x_n$ occurs in $t$. We define a function $\bar t:F^{a(t)} \to F$ by structural induction, as we now describe.
We remark that  $\bar t$ will essentially be the term function associated with $t$ and usually denoted by $t^\A$. However, our  definition of $\bar t$ is needed due to some minor technicalities involving the arities of term functions.

For each $i\in\mathbb{N}$  set $\bar x_i(w_1,\dots,w_i)=w_i$. If $t=\nu_c(s)$ for some $s$, then noting that $a(t)=a(s)$, let $\bar t=\nu_c^\A \circ \bar s$.  Finally, for $t=g(t_1,t_2)$, we set 
 $$\bar t=g^\A \circ \left(\bar t_1 \circ \pi^{a(t)}_{a(t_1)},\bar t_2 \circ \pi^{a(t)}_{a(t_2)}\right),$$
 which is well-defined, as $a(t) \ge a(t_1), a(t_2)$. With some abuse of terminology, we  will refer to all functions of the form $\bar t$ as term functions. 
 


Given a term $t \in T$, let $\nu(t)$ to be the set of  $c\in F$ such that  $\nu_c$ appears in $t$. We define the \emph{content} of $t$, denoted $C_t$,  by
\[C_t=\bigcup_{c\in \nu(t)}\{ z_i\in Z: z_i \mbox{ appears in the normal form of }c\}.\]
Note that
$C_t\subseteq Z$ and is finite. 

For each $t \in T$, we define $t^* \in \mathbb{N}$ by structural induction as follows: if $t=x_i$ for some $i$, then $t^*=i$, if $t=\nu_c(t_1)$ for some
$t_1 \in T$, then $t^*={t^*_1}$, and if $t=g(t_1,t_2)$ for some $t_1,t_2 \in T$, we set $t^*={t^*_2}$. It is easy to see that $t^*$ is the index of the variable that appears syntactically in the ``right-most" position of $t$ and clearly, $t^*\leq a(t)$.

The following lemma characterizes behaviour of the functions of the form  $\bar t$, by connecting them to the group structure on the underlying set $F$ of $\mathbb{A}$. This will be essential to our later arguments.

\begin{lemma}\label{lm:structure} Let $t\in T$ such that $a(t)=n$. 
 Then one of the following holds:
\begin{enumerate} 
\item\label{form1} there exist a $w \in FG(E \cup C_t)$ such that for all $\vec y \in F^n$, 
 $$\bar t(\vec y)= w y_{t^*};$$
\item 
there exists a function $f: F^n \to FG(E \cup C_t)$ such that for all $\vec y \in F^n$, 
$$\bar t(\vec y)= f(\vec y) y_{t^*}.$$
In addition, there are  sequences $z_{j_1}, z_{j_2},\ldots$ on $Z$, and $\vec \mu_1, \vec \mu_2,\ldots$ on $ (F_+)^n$ such that
\begin{enumerate}\label{form2}
\item $j_i \ne j_{i'}$ for $i \ne i'$, and
\item
$z_{j_i}$ appears in the normal form of $f(\vec \mu_i)$ with a positive exponent.
\end{enumerate}
\end{enumerate}
\end{lemma}
\begin{proof} We remark that if $f$ is as in Condition~(\ref{form2}), then, in particular, the image of $f$ is infinite, {\color{black} indeed, the image of $f$ restricted to
$(F_+)^n$ is infinite}. 

We prove the lemma by induction over the structure of $t$. 
If $t=x_i$ for some $i$, then
$n=a(t)=t^*=i$ and $\bar t(\vec y)=
\bar{x}_i(y_1,\hdots, y_i)=y_i=1y_{t^*}$ and the result holds with $w=1$ in  
Condition~(\ref{form1}).

Assume for induction that the result holds for any proper subterm of $t$. 

\noindent{\em Case (i)} Suppose first that
$t=\nu_\sigma (t_1)$ for some term $t_1$, so that $t^*=t_1^*$.
Then $n=a(t)=a(t_1)$ and
 by induction, $ t_1$ satisfies the conditions of the lemma.   

\noindent{\em Case (i)(a)} If $\bar t_1(\vec y)=w y_{t^*_1}$ for some {\color{black} $w \in FG(E\cup C_{t_1})$,} then 
 $\bar t(\vec y)= \sigma \bar t_1(\vec y)= \sigma w y_{t^*_1}=\sigma w y_{t^*}$. 
Moreover,  $\sigma w \subseteq FG(E \cup C_t)$,
as $C_t = C_{t_1} \cup C_\sigma$, where $
C_\sigma$ is the set of generators that appear in the normal form of $\sigma$.
Hence $\bar t(\vec y)$  satisfies Condition~(\ref{form1}).


\noindent{\em Case (i)(b)}  Now suppose that Condition~(\ref{form2}) holds
for $t_1$, so there exists 
{\color{black} $f_1: F^n \to FG(E \cup C_{t_1})$} and sequences $(z_{j_i})_{i\in\mathbb{N}}$ and $ (\vec \mu_i)_{i\in\mathbb{N}}$  as in (2),  such that  $\bar t_1(\vec y)= f_1(\vec y) y_{t^*_1}$. Then
$$\bar t(\vec y)= \sigma \bar t_1(\vec y)= \sigma f_1(\vec y) y_{t^*_1}=\sigma f_1(\vec y) y_{t^*}.$$


{\color{black} For each $i\in \mathbb{N}$ let  $w_i=f_1(\vec \mu_i)$
and put $ \sigma w_i =:\tau_i$.
The normal form of $w_i$ contains $z_{j_i}$ with a positive exponent, so
the normal form of $\tau_i$ will do so as well, unless $z_{j_i}$ cancels against a $z_{j_i}^{-1}$. 
But, considering $\sigma$, there are only finitely many indices $i$ for which $z_{j_i}^{-1}$ appears in the normal form of $\sigma$. It follows that for infinitely many values $i$, 
the element $ \tau_i\in F$ contains $z_{j_i}$ in its normal form with a positive exponent. }

Define {\color{black}$f: F^n \to F$} by  $f=\nu_\sigma^\A \circ f_1$. By the assumption on $f_1$, and as $C_t=C_{t_1} \cup C_\sigma$, we have $f:F^n\rightarrow  FG(E \cup C_t)$.

We obtain that $\bar t(\vec y)= f(\vec y)y_{t^*}$. It is easy to see that  $\bar t$ satisfies Condition~(\ref{form2}),  with  the sequences $(z_{j_i})_{i\in\mathbb{N}}$ and $(\vec \mu_i)_{i\in\mathbb{N}}$ obtained from the corresponding sequences for $\bar t_1$ by removing finitely many elements.  

\noindent{\em Case (ii)}  We now consider the case that $t=g(t_1,t_2)$ for some terms $t_1,t_2$. By induction the lemma holds for $t_1$ and $t_2$. Let $n_1=a(t_1), n_2=a(t_2)$, so that $n$
is the maximum of $n_1$ and $n_2$, and $t^*=t_2^*$. {\color{black} Notice that $C_t=C_{t_1}\cup C_{t_2}$.}

\noindent{\em Case (ii)(a)} Assume first that  Condition~(\ref{form2}) holds for
$t_2$, so that 
$$\bar t_2\left(\pi^{n}_{n_2}(\vec y)\right)=f_2\left(\pi^{n}_{n_2}(\vec y)\right) y_{t^*_2}$$ for some 
$f_2:F^{n_2}\rightarrow FG(E\cup C_{t_2})$, such that there are sequences $(z_{j_i})_{i\in\mathbb{N}}$ and    $(\vec \mu_i)_{i\in\mathbb{N}}$  as   in  (\ref{form2}).

Since $t^*=t_2^*$ we have
$$\bar t(\vec y)= h_{\bar t_1\left(\pi^{n}_{n_1}(\vec y)\right)\left(\bar t_2\left(\pi^{n}_{n_2}(\vec y)\right)\right)^{-1}} \bar t_2\left(\pi^{n}_{n_2}(\vec y)\right)=  h_{\bar t_1\left(\pi^{n}_{n_1}(\vec y)\right)\left(\bar t_2\left(\pi^{n}_{n_2}(\vec y)\right)\right)^{-1}} f_2\left(\pi^{n}_{n_2}(\vec y)\right) y_{t^*}.$$

Define $f:F^n\to F$ by 
$$f(\vec y)=h_{\bar t_1\left(\pi^{n}_{n_1}(\vec y)\right)\left(\bar t_2\left(\pi^{n}_{n_2}(\vec y)\right)\right)^{-1}} f_2\left(\pi^{n}_{n_2}(\vec y)\right).$$ 
We have that $\bar t (\vec y)=f(\vec y) y_{t^*}$, as required. Moreover, 
$f:F^n\rightarrow FG(E \cup C_t)$, by the conditions on $f_2$ and since  the image of $h$ lies in $E$.
Let   $\vec\mu_i'\in F_+^n$ be obtained by extending $\vec\mu_i$ to arity $n$ with $n-n_2$ arbitrary elements from $F_+$. By Condition (2) for $t_2$,  we have that
 $z_{j_i}$ appears in
the normal form of $f_2(\vec \mu_i)=w_i$ with a positive exponent. 
Now $$f(\vec \mu_i')= h_{\bar t_1\left(\pi^{n}_{n_1}(\vec \mu_i')\right)\left(\bar t_2(\vec \mu_i)\right)^{-1}}w_i.$$ 
By definition of $h$, the first factor is just an element of $E$, 
so in particular an element of $F_+$. It follows that the generator $z_{j_i}$ in $w_i$ cannot cancel, and hence appears in the normal 
form of $f(\vec \mu_i')$. 

Thus $\bar t$ satisfies Condition (\ref{form2})\ with $f$ and the sequences $(\vec \mu_i')_{i\in\mathbb{N}}$ and $ (z_{j_i})_{i\in\mathbb{N}}$. 


\noindent{\em Case (ii)(b)} For our final case we assume  that $\bar t_2\left(\pi^n_{n_2}(\vec y)\right) = w_2 y_{t^*_2}$, for some $w_2 \in FG(E \cup C_{t_2})$.
We make four further case distinctions.

\begin{enumerate}
\item $\bar t_1$ satisfies Condition (\ref{form1}) and $t^*_1=t^*_2$. {\color{black}
We have for $\vec u\in F^{n_1}$ that $t_1(\vec u)=w_1u_{t_1^*}$ for some $w_1$, 
and then for $\vec y\in F^n$ we see that
$\bar t_1\left(\pi^n_{n_1}(\vec y)\right)= w_1 y_{t^*_1}$, and 
\[\bar t(\vec y)= h_{\bar t_1\left(\pi^n_{n_1}(\vec y)\right)
(\bar t_2\left(\pi^n_{n_2}(\vec y)\right)^{-1}}\bar t_2\left(\pi^n_{n_2}(\vec y)\right)=
h_{w_1y_{t_1^*}(w_2y_{t_2}^*)^{-1}} w_2 y_{t^*_2}=
h_{w_1w_2^{-1}} w_2 y_{t^*}\]
as $t^*_1=t^*_2$.  Thus $\bar t$  also satisfies Condition (\ref{form1}), as 
$h_{w_1w_2^{-1}} w_2 \in FG(E \cup C_t)$.}

\item $\bar t_1$ satisfies Condition~(\ref{form2}) with respect to {\color{black} $f_1:F^{n_1}\rightarrow FG(E\cup C_{t_1})$} and  $t^*_1=t^*_2$. 


In this case  
$$\bar t(\vec y)=h_{f_1\left(\pi^n_{n_1}(\vec y)\right)w_2^{-1}}w_2 y_{t^*}.$$ We claim that $\bar t$ satisfies Condition (\ref{form2}). Let {\color{black} $f$} be given by $$f(\vec y)=h_{f_1\left(\pi^n_{n_1}(\vec y)\right)w_2^{-1}}w_2.$$
Then $f:F^n\rightarrow FG(E\cup C_t)$ by the same argument as above, so it remains to construct appropriate sequences $(\vec \mu_i)_{i\in\mathbb{N}}$ and $(z_{j_i})_{i\in\mathbb{N}}$. 

By the remark at the beginning of this proof,  $f_1\left(\pi^n_{n_1}\left(F_+^{n}\right)\right)$
is infinite and hence so is $f_1\left(\pi^n_{n_1}\left(F_+^{n}\right)\right)w_2^{-1}$. 
The function $h$ is injective and maps into $E$, so it follows that   $h_{f_1(\pi^n_{n_1} (\vec x))w_2^{-1}}$ takes on infinitely many values $z_{j_i}$ in $E$ as $\vec x$ runs over $F^{n}_+$. 
Only finitely many of these values can cancel against a generator in the normal form of $w_2$. The existence of $(\vec \mu_i)_{i\in\mathbb{N}}$ and $(z_{j_i})_{i\in\mathbb{N}}$ follows.

\item $\bar t_1$ satisfies Condition (\ref{form1}) and $t_1^* \ne t_2^*$. In this case
we have that
 $\bar t_1(\vec y)= w_1 y_{t_1^*}$ for some $w_1 \in FG(E \cup C_{t_1})$, for all $\vec y \in F^{n_1}$, and thus
 $$\bar t(\vec y)=h_{w_1y_{t_1^*} (y_{t_2^*})^{-1} w_2^{-1}} w_2 y_{t^*}.$$ 
 {\color{black} Consider the set $P\subset Z^2$ of pairs $(u_{1},u_{2})$ for which $u_{1} \ne u_{2}$,
 and neither $u_{1},u_{2}$ nor their inverses appear in the normal forms
of $w_1$ or $ w_2$; clearly $P$ is  infinite. For any $(u_1,u_2)\in P$, the normal form of $w_1 u_{1} u_{2}^{-1} w_2^{-1}$ contains the subexpression $u_1u_2^{-1}$, and these are the only occurrences of $u_1$ and $u_2$ in the 
  normal form. It follows that $w_1 u_{1} u_{2}^{-1} w_2^{-1}$
  takes on only distinct and hence infinitely many elements of $F$ as $(u_{1},u_{2})$ runs through $P$. As 
 $h$ is injective, $h_{w_1 u_{1} u_{2}^{-1} w_2^{-1}}$
   also takes on infinitely many values from $E$ as  $(u_{1},u_{2})$ runs through $P$. Only finitely many of those values can cancel against generators from the normal form of 
   $w_2$. Removing the corresponding pairs from $P$ we see that
 $\bar t$ satisfies Condition (\ref{form2}) with respect to 
$f:F^{n}\rightarrow FG(E\cup C_{t})$, where
  $$f(\vec y)= h_{w_1y_{t^*_1} y_{t^*_2}^{-1} w_2^{-1}} w_2,$$
with the $\vec \mu_i\in F_+^n$ being chosen so that  $\vec \mu_i(t^*_1)=u_{1}^{(i)}$, $\vec \mu_i(t^*_2)=u_{2}^{(i)}$ and with arbitrary elements of $F_+$ in all other coordinates, where $(u_{1}^{(i)}, u_{2}^{(i)})$ runs over a cofinite subset of $P$ and 
$\left(z_{j_i}\right)_{i\in\mathbb{N}}=(h_{w_1 u_{1}^{(i)} (u_{2}^{(i)})^{-1} w_2^{-1}})_{i\in\mathbb{N}}.$}
 \item $\bar t_1$ satisfies Condition (\ref{form2}) with respect to  {\color{black} $f_1:F^{n_1}\rightarrow FG(E\cup C_{t_1})$, }
and  $t^*_1\neq t^*_2$.

  This case is similar to the previous one. We have that 
 $$\bar t(\vec y)=h_{f_1\left(\pi_{n_1}^n(\vec y)\right)y_{t^*_1} y_{t^*_2}^{-1} w_2^{-1}} w_2 y_{t^*}.$$
 {\color{black}Let $P\subset Z^2$ be the set of pairs $(u_{1},u_{2})$ for which 
 $u_{1},u_{2} \notin E\cup C_t$, $u_{1} \neq u_{2}$  and
 neither $u_{1}$, $u_{2}$ nor their inverses appear in the normal form
of $w_2$; clearly $P$ is infinite.
  If $\vec z \in F^n_+$ and $(u_{1},u_{2}) \in P$, then  in the expression 
  $$f_1\left(\pi_{n_1}^n(\vec z)\right) u_{1} u_{2}^{-1} w_2^{-1},$$ $u_{1}$ and $u_{2}^{-1}$ cannot cancel against any generators from the normal forms of $w_2^{-1}$ and $f_1\left(\pi_{n_1}^n(\vec z)\right)$, in the latter case because $f_1$ maps into $F(E\cup C_t)$.}
Arguing as previously we see that $\bar t$ satisfies Condition \ref{form2}.
  
  \end{enumerate}
By structural induction, the lemma holds for all terms $t\in T$.
  \end{proof}

We are ready show our main result.

\begin{theorem}\label{thm:main} The independence algebra $\A$ does not satisfy the distributivity property.
\end{theorem}  
\begin{proof}
By way of contradiction, assume that  $W$ is set of functions that generate the clone of $\A$ and witness the distributivity property. 
The clone of $\A$ contains the function $g^\A$. By the first three conditions of our choice of $h$, calculation shows that $g(z_1,z_2)=z_6z_2, g(z_3,z_2)=z_8z_2,$ and  
$g(z_1,z_4)=z_{10}z_4$. 
Combined, these results show that $g$  depends on both of its arguments. 

It follows that $W$ must contain an operation $v$ that depends on more than one argument, for otherwise the entire clone of $\A$ would consist of functions that are essentially unary. As $v$ is in the clone of $\A$, it  is a composition of $\A$-operations and projections, and it is easy to see that such $v$ must have the 
form $\bar t \circ \pi^n_m$ for some $t\in T$. 
Moreover, $W \setminus \{v\} \cup \{\bar t\}$ also generates the clone of $\A$ and witnesses the distributivity property. Thus, we may assume that $v=\bar t$ for a term $t \in T$; since $v$ depends on at least two variables, so does $\bar t$. 

 Such a $\bar t$ must satisfy one of the two conditions from Lemma \ref{lm:structure}. As the first condition implies that $\bar t$ only depends on one variable, we must have instead that $\bar t (\vec y)= f(\vec y) y_{t^*}$, where
 $f$ is as in Condition (\ref{form2}) of Lemma \ref{lm:structure}.  Let $(\vec \mu_i)_{i\in\mathbb{N}}$ and $(z_{j_i})_{i\in\mathbb{N}}$ be the sequences associated to $f$ 
 satisfying (a) and (b) of  Lemma \ref{lm:structure}, and set $w_1=f(\vec \mu_1)$.

 We have that $ \bar t(\vec \mu_1)= f(\vec \mu_1) \mu_1^*$ - where $\mu_1^*$ is the $t^*$-th entry of $\vec \mu_1$.  Choose $a \in Z \setminus (E \cup C_t)$. As we assume that $\A$ satisfies the distributivity property, and $W$ is a witness of it, we have that 
 $\nu_a(\bar{t}(\vec \mu_1))=\bar{t}(\nu_a(\vec \mu_1))$, where, with abuse of notation,
 $\nu_a(\vec \mu_1)=a\vec \mu_1$ is the element of $F^n$
 obtained by multiplying every coordinate of $\vec \mu_1$ by $a$ on the left. Hence $ af(\vec \mu_1) \mu_1^*= f(a \vec \mu_1) a \mu_1^*$ and so $  af(\vec \mu_1) = f(a \vec \mu_1) a .$
Now, $f(a \vec \mu_1)$ is in the image of $f$ which is contained in $FG(E \cup C_t)$ by Lemma \ref{lm:structure}. As $a \notin FG(E \cup C_t)$, the normal form of 
$f(a \vec \mu_1)a$, and hence the normal form of $af(\vec \mu_1)$, 
will end with $a$.  However, $f(\vec \mu_1)$ is also an element of $FG(E \cup C_t)$, and so its normal form does not contain $a$. It follows that $f(\vec \mu_1)=1$. However, by Lemma \ref{lm:structure}, the normal form of $f(\vec \mu_1)$ contains the generator $z_{j_1}$,  a contradiction.
\end{proof}

\section{Conditions for $\End{\B}$ to be a right order in $\End{\A}$}\label{sec:right}

{\color{black} Throughout this section, let $\B$ be a stable basis algebra of finite rank $n$
satisfying the distributivity condition. Thus far we  have not been explicit about the properties satisfied by $\B$, but for the convenience of the reader we give some brief details.
All of these, and further information, can be found in \cite{basisi}.
The relation $\prec$ is defined by 
\[a\prec W\Leftrightarrow a\in \langle \emptyset\rangle_{\mathbb{B}}
\mbox{ or } \langle \{ a\} \rangle _{\mathbb{B}}\cap \langle W\rangle_{\mathbb{B}}
\neq \langle \emptyset\rangle_{\mathbb{B}}\]
for $a\in B$ and $W\subseteq B$. We say that $W$ is {\em pure closed} if
\[a\prec W\Leftrightarrow a\in W.\]
Any subset of $V$ of $B$ is contained in smallest pure closed subset $\pc(V)$;
indeed, $\pc(V)$ is a subalgebra with a basis that can be extended to a basis
of $B$. To say $X$ is a  {\em basis} of a subalgebra $C$ means that 
$C=\langle X\rangle_{\mathbb{B}}$, every map from $X$ to $B$ lifts to a (unique) morphism
from $C$ to $B$, and $x \not\prec X\setminus \{ x\}$, for any $x\in X$.
The {\em rank} of a subalgebra $U$ is the cardinality of a basis of its pure closure
$\pc(U)$. Moreover, since $B$ has finite rank $n$, any subalgebra with rank $m\leq n$  itself has a basis of $m$ elements. However, it is only the pure closed subalgebras that have bases that are extendable to bases of $B$. Since  $\B$ has  rank $n$, we choose and fix a  basis $\{b_1,\dots,b_n\}$. 

As pointed out at the beginning
of Section~\ref{sec:example}, the monoid $\mathbb{T}$ of non-constant unary term operations is cancellative and left Ore. It therefore has a group of left quotients $\G$ by \cite[Theorem 1.24]{cp}. Further, the elements of $\mathbb{T}$ are injective (as maps from
$B$ to $B$). Another fact of which we make use is that if $u(x)$ is a unary term operation of $\B$, and $u(b)\in \langle \emptyset\rangle_{\mathbb{B}}$ for some
$b\notin \langle \emptyset\rangle_{\mathbb{B}}$, then $u(x)$ is a constant map with image
 $u(b)\in \langle \emptyset\rangle_{\mathbb{B}}$.

Finally, for any $\alpha\in \en \B$, the {\em rank} of $\alpha$ is defined to be the rank of the subalgebra $\im \alpha$; this always exists. }

In \cite{gouldI} Gould constructs an independence algebra $\A$ such that $\B$ is a reduct of $\A$, and such that $\en(\B)$ is a left order  (in the sense of
Definition~\ref{defi:order}) in $\en(\A)$. 
For convenience we gather here some essential information concerning the construction of
$\A$ from $\B$. Further details may be found in \cite{gouldI}.

Let 
$\Sigma=T\times B$ and define a relation $\sim$ on $\Sigma$ by the rule that
\[(a,c)\sim (b,d)\Leftrightarrow xa=yb\mbox{ and }x(c)=y(d)\mbox{ for some }x,y\in T;\]
{\color{black}for clarity, here $xa,yb$ are products in the {\em monoid} $\mathbb{T}$ and $x(c),y(d)$ are the values of $x,y$ acting on $c,d\in B$, respectively.}
Then $\sim$ is an equivalence relation and $A=\Sigma/\sim$ is the underlying set of an independence algebra $\A$. Further, $\B$ is  a reduct of $\A$, where $\B$ embeds in
$\A$ under $b\mapsto [1,b]:=[(1,b)]$.

\begin{prop}\label{prop:facts}\cite{gouldI} Let $\B$ and $\A$ be as above. Then 
$$\{[1,b_1],\dots, [1,b_n]\}$$ 
is a basis for $\A$ so that $\A$ and   $\B$  have the same rank. The map $\theta\mapsto \bar{\theta}$ from
$\en(\B)$ to $\en(\A)$ embeds $\en(\B)$ as a left order in
$\en(\A)$, where for any $[a,b]\in A$ we have
$[a,b]\bar{\theta}=[a,b\theta]$. Further, $\theta$ and $\bar{\theta}$ have the same rank.
\end{prop}

Before addressing the main question of this section, we first confirm the way in which 
$\End (\B)$ sits inside $\End(\A)$ in the general case. Recall from \cite{gould ab}
that a left order $\mathbb{S}$ in $\mathbb{Q}$ is {\em fully stratified} if for any
$a,b\in S$ we have
\[a\,\leqars\, b\mbox{ in }\mathbb{S}\mbox{ if and only if }a\,\leqar\, b\mbox{ in }\mathbb{Q}\]
and
\[a\,\leqels\, b\mbox{ in }\mathbb{S}\mbox{ if and only if }a\,\leqel\, b\mbox{ in }\mathbb{Q}.\]
The relations $\leqar$ and $\leqel$ above are the pre-orders associated with Green's relations $\ar$ and $\el$;  $\leqars$ and $\leqels$ above are the pre-orders associated with the larger relations $\ars$ and $\els$. We give further details as and when we use them; the reader may consult \cite{gould ab}.  These relations are important for the study of
$\en(\B)$ due to the following result.

\begin{prop}\label{prop:greens} \cite{basisii,gould} Let $\mathbb{C}$ be an independence algebra and
$\mathbb{D}$ a basis algebra, let $\alpha,\beta\in \en(\mathbb{C})$ and
$\gamma,\delta\in \en(\mathbb{D})$. Then  

\[\alpha\,\leqar\, \beta\mbox{  if and only if }\ker\beta\subseteq \ker \alpha;\]
\[\alpha\,\leqel\, \beta\mbox{  if and only if }\im\alpha\subseteq \im\beta;\]
\[\gamma\,\leqars\, \delta\mbox{  if and only if }\ker\delta\subseteq\ker\gamma\]
and
\[\gamma\,\leqels\, \delta\mbox{ if and only if }\pc(\im\gamma)\subseteq\pc(\im\delta).\]
\end{prop}

\begin{theorem}\label{rightord}  Let $\B$ be a stable basis algebra of finite rank $n\geq 1$ such that $\B$ satisfies the distributivity condition. The monoid $\End(\B)$ is  a 
fully stratified left order  in $\End(\A)$. 
\end{theorem}
\begin{proof} We are required to show that for any $\alpha,\beta\in \End(\B)$ we have
that $\alpha\,\leqels\,\beta$ in $\End(\B)$ if and only if $\overline{\alpha}\,\leqel\,\overline{\beta}$ in $\End(\A)$
and dually, $\alpha\,\leqars\,\beta$ in $\End(\B)$ if and only if $\overline{\alpha}\,\leqar\,\overline{\beta}$ in $\End(\A)$.

The proof of the first statement is inherent in  the proof of \cite[Proposition 5.2 (ii)]{gould}, although not explicitly stated. We concentrate on the second; according to Proposition~\ref{prop:greens}  it is sufficient to show  that if 
$\ker \beta \subseteq \ker \alpha$, then 
$\ker\overline{\beta}\subseteq \ker\overline{\alpha}$.

Suppose therefore that $\ker \beta \subseteq \ker \alpha$, and consider
$[a,u],[b,v]\in A$ with $[a,u]\overline{\beta}=[b,v]\overline{\beta}$.
{\color{black} By Proposition \ref{prop:facts}} we have that
$[a,u\beta]=[b,v\beta]$ so that by the definition of $\sim$, there exist
$c,d\in T$ with $ca=db$ and $c(u\beta)=d(v\beta)$ so that
$(c(u))\beta=(d(v))\beta$. Since $\ker\beta\subseteq \ker\alpha$ 
we have $c(u\alpha)=(c(u))\alpha=({d}(v))\alpha={d}(v\alpha)$ and then
$[a,u]\overline{\alpha}=[a,u\alpha]=[b,v\alpha]=[b,v]\overline{\alpha}$,
as required.
\end{proof}

We have remarked that any $a\in T$ is one-one as a map, and certainly
\[a|_{\langle \emptyset\rangle_{\mathbb{B}}}:\langle \emptyset\rangle_{\mathbb{B}}\rightarrow \langle \emptyset\rangle_{\mathbb{B}}.\]
 
\begin{defi}\label{defn:CI}  The algebra $\B$ satisfies the {\em
Constant Isomorphism Property} (CI) if: 
\[a|_{\langle \emptyset\rangle_{\mathbb{B}}}:\langle \emptyset\rangle_{\mathbb{B}}\rightarrow \langle \emptyset\rangle_{\mathbb{B}}\]
is {\em onto}, hence an  {\em isomorphism} of the constant subalgebra $ \langle \emptyset\rangle_{\mathbb{B}}$ of $\mathbb{B}$.
\end{defi}

We introduced (CI) for the following purpose.

\begin{theorem}\cite[Theorem 6.2]{gouldI}\label{thm:straight} Let $\B$ be a stable basis algebra of finite rank satisfying the distributivity condition. Then $\en (\B)$ is a straight left order in $\en(\A)$ if and only if $\B$ satisfies the (CI) and (if $n\geq 2$), $\mathbb{T}$ is right Ore.
\end{theorem}

From \cite[Theorem 1.24]{cp}, if $\mathbb{T}$ is right Ore, then it has a group of right quotients and it is easy to see in this case that $\G$ is a group of (two-sided) quotients
of $\mathbb{T}$. The paper \cite{gouldI} leaves open the question of whether
the fact that $\G$ is a group of quotients of $\mathbb{T}$ forces $\End(\B)$ to be a
right and hence (two-sided) order in $\End(\A)$. 
 {\color{black} The aim of this section is to determine the conditions  under which $\End(\B)$ is a right order in $\End(\A)$, thus answering the question posed in the positive for $n\geq 2$ and establishing {\bf Result 2}.}
 
If the rank $n$ of $\B$  is $0$, then $\A=\B=\langle \emptyset \rangle$ 
and so  $\End(\A)=\{I_A\}=\End(\B)$ is a one element group, and our results hold trivially.
 We will therefore restrict to $\B$ with positive rank. 
Via a series of lemmas we now prove:

\begin{theorem}\label{rightord}  Let $\B$ be a stable basis algebra of finite rank $n\geq 1$ such that $\B$ satisfies the distributivity condition. The monoid $\End(\B)$ is  a right order in $\End(\A)$ if and only if $\mathbb{T}$ is right Ore and $\B$ satisfies (CI).
\end{theorem}

In fact, we will show slightly more, demonstrating that in the right quotient decomposition $\alpha =\bar \gamma \bar\beta^{\#}$, $\bar\beta$ can be chosen as an automorphism of $\A$. 
Notice that the conditions given in Theorem~\ref{rightord}  imply those of Theorem~\ref{thm:straight}, and coincide if $n\geq 2$.

\begin{lemma} \label{minilemma} Suppose $\mathbb{T}$  is right Ore and satisfies (CI). Let $\{s_1, \dots, s_k\}$ be a finite non-empty set of terms in the language of $\B$ 
over the variables set $\{x_1,\dots,x_n\}$ and let $a \in T$. Then 
there exists an endomorphism $\theta_a$ of $\B$  such that $s_{i}^\B(b_1,\dots,b_n)\theta_a\in \im a$ for $1\le {i}\le k$, where we interpret all $s_i$ as $n$-ary operations. Moreover, $\theta_a$ 
satisfies $b_{\ell}\theta_a=a(r(b_{\ell}))$ for $1\le \ell\le n$ and some $r\in T$.
\end{lemma}
\begin{proof} We have remarked that under these hypotheses $\G$ is the group of quotients of $\mathbb{T}$. 

Let $m$ be the number of basic, non-nullary term operations appearing in the terms $s_1,\dots, s_k$, counted with multiplicity. We prove the lemma by induction on $m$.

If $m=0$ then all $s_i$ are either variables or nullary operation symbols. By reordering the $s_i$ as necessary  we may assume that for some $1\le t\leq k+1$ we have
\[ s_{i}^\B(b_1,\dots, b_n)=
\left\{ \begin{array}{ll}b_{j_{i}}& 1\le {i}< t\\
e_{i}& t\le {i}< k+1,\end{array}\right.\]
for some  constants $e_{i}$.  By (CI), $a$ acts as an isomorphism on $ \langle \emptyset\rangle_\B$, and so $e_{i}=a(f_{i})$, 
for some $f_{i} \in\langle \emptyset\rangle_\B$. 

Define $\theta_a$ by $b_i\theta_a=a(b_i)$, then 
$$s_{i}^\B(b_1,\dots,b_n)\theta_a=b_{j_{i}}\theta_a=a(b_{j_{i}})$$ for $1\le {i} < t$ and 
$$s_{i}^\B(b_1,\dots,b_n)\theta_a=e_{i}\theta_a=e_{i}=a(f_{i})$$ for $t\le {i}< k.$
Thus $\theta_a$ satisfies the conditions of the lemma with $r$ being the identity.

Now suppose that $0<m$, and that the result holds for all $m'$ with
$0\leq m'<m$. We may reorder the $s_i$ such that for some integers $1\leq k_1\leq  k_2\leq k_3\leq k_4\leq k+1$ we have
\[ s_i =\left\{ \begin{array}{ll} x_{j(i)} & \mbox{ for }1\le i < k_1\\
v_i(h_1^i,\dots ,h_{l(i)}^i)& \mbox{ for }k_1\le i <k_2\\
u_i(h_i)&\mbox{ for }k_2\le i < k_3\\
c_i(h_i)&\mbox{ for }k_3\le i < k_4\\
d_i& \mbox{ for }k_4\le i <k+1,\end{array}\right.\]
where 
$1\leq j(i)\leq n$, the $v_i$ are {\color{black} basic $l(i)$-ary operation symbols} where $l(i)\geq 2$,  the $u_i,c_i$ are unary function symbols with 
 $u_i^\B=a_i \in T$ and 
$c_i^\B\notin T$, the $d_i$ are  nullary operation symbols, and the $h_i$ and $h_j^i$ are arbitrary terms. {\color{black}  Our convention is that if, for example, $1=k_1=k_2$, then there are no instances of $s_i$ of the first two kinds. }

For  $k_3\leq i< k$ we have that, as $c_i^{\mathbb{B}}\notin T$, 
$s^\B(b_1,\dots,b_n)=g_i\in\langle \emptyset\rangle_{\mathbb{B}}$, so that by the same argument used for $m=0$, (CI) gives that  $ g_i\in \im a$. 

If $k_1=k_2=k_3$, then let $b_{\ell}\theta_a=a(b_{\ell})$ for $1\leq \ell\leq n$, so that the conditions of the lemma hold with $r$ the identity. Otherwise, since 
$\T$ is right Ore, we proceed as follows. First, by applying the common denominator theorem in  the group $\G$ to the elements ${a}^{-1}a_i$ for
$k_1\le i < k_2$, we may find $p_i, u \in T$, such that in $\G$,
$a^{-1}a_i=p_iu^{-1},$
and hence in $\T$, we have $a_i u=ap_i$. If $k_1=k_2$, let $u\in T$ be chosen arbitrarily. Again from  $\T$ being right Ore, there are $p, v \in T$ for which $up=av$.  

We now apply our induction hypothesis to all the terms, $x_{j(i)}$, $h_j^i$ and $h_i$ with $av$ in place of $a$ (note that we must have at least one such term).
It follows that there exists  $\phi_{av} \in \End(\B)$ and $z_i',z_j^i, z_i \in B$, for which
$$\left(x_{j(i)}^{\B}(b_1,\dots,b_n)\right)\phi_{av}= av(z_i'),$$
$$\left(h_j^{i\B}(b_1,\dots,b_n)\right)\phi_{av}= av(z_j^i),$$
$$\left(h_i^{\B}(b_1,\dots,b_n)\right)\phi_{av}= av(z_i),$$
and that, moreover, $b_{\ell}\phi_{av}=av(r'(b_{\ell}))$ for   $1\leq\ell\leq n$ and some $r'\in T$.

For $1\le i< k_1$,
\[\begin{array}{rcl}
s_i^\B(b_1,\dots, b_n)\phi_{av}&=&\left(x_{j(i)}^{\B}(b_1,\dots,b_n)\right)\phi_{av}\\&=& av(z_i'),
\end{array}\]
For $k_1\le i< k_2$,
\[\begin{array}{rcl}
s_i^\B(b_1,\dots, b_n)\phi_{av}&=&
\left(v_i^\B\left(h_1^{i\B}(b_1,\dots,b_n), \dots, h_{l(i)}^{i\B}(b_1,\dots, b_n)\right)\right)\phi_{av}\\
&=&v_i^\B(av(z_1^i), \dots, av(z_{l(i)}^i))\\
&=&av_i^\B(v(z_1^i), \dots, v(z_{l(i)}^i)),
\end{array}\]
where the last equality follows from the distributivity condition. 

For $k_2\le i< k_3$ we have $$\left(s_i^\B(b_1,\dots,b_n)\right)\phi_{av}=\left(u_i^\B(h_i^\B(b_1,\dots,b_n))\right) \phi_{av}
=\left(a_i(h_i^\B(b_1,\dots,b_n))\right)\phi_{av}$$
$$=a_i(av(z_i))=a_i(up(z_i))=ap_ip(z_i).$$
Since  $\phi_{av}(b_{\ell})=avr'(b_{\ell})$ for some $r' \in T$ and $1\leq \ell\leq n$,  the result holds with $\theta_a=\phi_{av}$ and $r=vr'$. 
\end{proof}

\begin{theorem}\label{ro}
If $\T$ is right Ore and (CI) holds, then $\End(\B)$ is a right order in $\End(\A)$. Moreover every $\alpha \in \End(\A)$ may be written in the form
$\alpha=\bar \gamma \bar\beta^{-1}$ for some $\gamma, \beta \in \End(\B)$ with $\bar \beta\in\Aut \A$. 
\end{theorem}
\begin{proof} 
From \cite{gouldI}, $\End(\B)$ is a left order in $\End(\A)$, so that certainly 
every square-cancellable element of $\End(\B)$ lies in a subgroup
 of $\End(\A)$. 
Let $\alpha \in \End(\A)$. It remains to find $\beta, \gamma \in \End(\B)$ such that $\alpha=\bar\gamma\bar\beta^{\#}$.
%

By Proposition~\ref{prop:facts} we have that  
$$\{[1,b_1],\dots, [1,b_n]\}$$ is a basis for $\A$. Choose $\mu_i \in T$, $d_i \in B$ 
such that $[1,b_i]\alpha=[\mu_i,d_i]$ for $1\leq i\leq n$.   

Consider the elements $\mu_i ^{-1} \in G$. By the common denominator theorem, there exists $\mu, \nu_i \in T$ such that 
$\mu_i^{-1}=\mu^{-1}\nu_i$ in  $\G$ for $1\leq i\leq n$.   It follows that $\mu=\nu_i\mu_i$, and thus 
from the definition of $\sim$, we have $[\mu_i, d_i] = [\nu_i \mu_i, \nu_i(d_i)]=[\mu, \nu_i(d_i)]$.
Putting $c_i := \nu_i(d_i)\in B$ for $1\leq i\leq n$  we have that $[1,b_i]\alpha=[\mu, c_i]$.

Now let $t_1, \dots, t_n$ be $n$-ary terms such that $c_i=t_i^\B(b_1, \dots,b_n)$
for $1\leq i\leq n$.  
By Lemma \ref{minilemma}, there exists an endomorphism $\theta_{\mu}$ of $\B,z_i \in B$
and $r\in T$ such that $t^\B_i(b_1,\dots,b_n)\theta_{\mu}=\mu (z_i)$  and
 $b_\ell\theta_\mu=\mu(r(b_\ell))$ for $1\leq \ell\leq n$.
Define $\beta,\gamma\in \End(\B)$ by $\beta=\theta_{\mu}$ and  $b_i\gamma=z_i$. Using Proposition \ref{prop:facts}, 
we have that  for all $i$,
$$[1,b_i]\alpha\bar\beta=[\mu,c_i]\bar\beta=[\mu,c_i\beta]=[\mu,c_i\theta_{\mu}]=[\mu,\mu (z_i)]=[1,z_i]=[1,b_i \gamma]=[1,b_i]\bar\gamma.$$
It follows that $\alpha \bar \beta=\bar\gamma$.
From the fact $\{ b_1,\hdots ,b_n\}$ is a basis for $\mathbb{B}$, we see that
\[ \mu(r(b_i))\not\prec \{ \mu(r(b_1)), \hdots, \mu(r(b_{i-1})),\mu(r(b_{i+1})),
\hdots, \mu(r(b_n)\}\] for any $1\leq i\leq n$.  It follows that $\beta$ has rank $n$, and hence by Proposition~\ref{prop:facts}, so does $\bar\beta$. But then $\im\, \bar\beta=A$, which means 
that $\bar \beta$ is a unit of $\A$ by \cite[Proposition 3.2]{gould}. Hence $\alpha=\bar \gamma \bar\beta^{-1}=\bar \gamma \bar\beta^{\sharp}$, as required. 
\end{proof}

We now tackle the converse to Theorem~\ref{ro}.

\begin{lemma}\label{CILemma} Suppose that $\End(\B)$ is a right order in $\End(\A)$. Then (CI) holds in $\B$.
\end{lemma}
\begin{proof} Let  $a \in T$ and  $c \in \langle \emptyset \rangle_{\mathbb{B}}$. We need to show that $c$ is in the image of $a$.

As $[1,b_1]$ is in a basis of $\A$ there exists an 
$\alpha \in \End(\A)$ such that $[1,b_1] \alpha=[a,c]$. As $\End (\B)$ is a right order in $\End (\A)$, there are $\beta, \gamma \in \End(\B)$ such that
$\alpha=\bar \gamma \bar \beta^{\#}$.
Note that
$$\alpha=\bar \gamma \bar \beta^{\#}=\bar \gamma \left(\bar \beta \bar \beta^{\#} \bar \beta^{\#}\right)=\left(\bar{\gamma} \bar  \beta\right) \left(\bar \beta \bar\beta\right)^{\#}=
\left(\overline{\gamma}\,\overline{ \beta}\right)\left(\overline{\beta}^2\right)^{\#}.$$
By potentially replacing $\gamma$ and $\beta$ with ${\gamma \beta}$ and $\beta^2$, respectively, we may assume 
that $\bar{\gamma}= \bar{\gamma}\bar{\beta}^{\sharp}\bar{\beta}$
and, with this assumption, we obtain $\bar \gamma =\alpha\bar \beta$.

From 
$[1,b_1] \alpha=[a,c]$ we have that  $$[1,b_1]\bar \gamma= 
[1,b_1]\alpha\bar \beta=[a,c]\bar \beta$$
giving
 $$ [1,b_1\gamma]=[a, c\beta] =[a,c].$$
Hence there exists $u,v \in T$ such that $ua=v1$ and $u(c)=v(b_1\gamma)$. Then $u(c)=ua(b_1 \gamma)$ and as $u$ is injective, we have that $c=a(b_1\gamma)$. By 
comments at the beginning of this section,  we have that $b_1\gamma\in \langle \emptyset \rangle_{\mathbb{B}}$. Thus 
(CI) holds in $\B$.
 \end{proof}

 \begin{theorem} \label{lr} Suppose that $\End(\B)$ is a right order in $\End(\A)$. Then $\T$ is right Ore.
 \end{theorem}
 \begin{proof} Let $p,q \in T$. 
 Define $\alpha\in \End(\A)$ by $[1,b_i] \alpha= [p, q (b_1)]$
 for $1\leq i\leq n$. 
 
 By \cite[Lemma 4.10]{gouldI} the rank of  $\alpha$ is 1. Since $\End(\B)$ is a right order in $\End(\A)$, there are $\gamma, \beta  \in \End(\B)$ such that $\alpha=\bar\gamma\bar\beta^{\#}$. As in the proof of Lemma \ref{CILemma}, we may assume that $\alpha \bar \beta =\bar \gamma$. Hence for $i=1,\dots,n$,
 $$[1,b_i\gamma]=[1,b_i]\bar \gamma=[1,b_i]\alpha \bar \beta =[p, q(b_1)]\bar \beta
 =[p, (q(b_1)) \beta] =[p, q(b_1 \beta)].$$

 Hence there exist $a,b \in T$ such that $ap=b$ and $aq(b_1 \beta)=b(b_i \gamma)$,
 giving $aq(b_1 \beta)=ap(b_i \gamma)$. By injectivity of $a$, we obtain 
 $q(b_1 \beta)=p(b_i \gamma)$.
 
 Let $u=b_1 \beta$, $v_i=b_i \gamma$. It follows that $q(u)=p(v_i)$ and so 
 $\pc(\{u\})=\pc(\{v_i\})=W$ (say) for $1\leq i\leq n$. 
 As it is an element of a basis, we have that $[1,b_1]\notin  \langle \emptyset\rangle_{\mathbb{A}}$. By \cite[Lemma 4.10]{gouldI}, $[p, q b_1] \notin \langle \emptyset\rangle_{\mathbb{A}}$.
 From $\alpha=\bar\gamma\bar\beta^{\#}$ 
 we have that 
 $\alpha=\alpha\bar\beta\bar\beta^{\#}$ and so
 $\bar\beta\bar\beta^{\#}$ acts as the identity on 
 $\im(\alpha)=\langle[p, q b_1]\rangle_{\mathbb{A}}$, giving  $[p, q(b_1 \beta)]=[p, q( b_1)] \bar\beta \notin \langle \emptyset\rangle_{\mathbb{A}}$. 
Once again by \cite[Lemma 4.10]{gouldI}, we obtain that 
$q(b_1\beta)\notin  \langle \emptyset\rangle_{\mathbb{B}}$ and hence 
by \cite[Proposition 2.4]{gouldI} we also have $ b_1 \beta \notin  \langle \emptyset\rangle_{\mathbb{B}}$. Thus $W$ has rank $1$. It follows that $W$ has a one-element basis, so $W=\langle w_1\rangle_{\mathbb{B}}$ for some $w_1 \in B$ where 
 $\{ w_1,\hdots, w_n\}$ is a basis of $B$. Hence there are $h, k \in T$, such that
 $h(w_1)=u $, $k(w_1)=v_1$, and so $q h(w_1)= p k(w_1)$. 
 Note that $\{w_1\}$ can be extended to a basis of $B$. It follows that  for any $b' \in B$, there is an endomorphism $\tau$ with $w_1 \tau=b'$. Clearly then $q h= p k$ and the result follows.
 \end{proof}
Theorem \ref{rightord} now follows directly from Theorem \ref{ro}, Lemma \ref{CILemma}, and Theorem \ref{lr}.

\begin{example} \label{ex:rings}
Let $\mathbb{R}$ be an integral domain and $\mathbb{M}$ an $n$-generated ($n\in\mathbb{N}$) free left
module over $\mathbb{R}$ such that $\mathbb{M}$ is a stable basis algebra; certainly we must have
that $\mathbb{R}$ is left Ore, since the non-zero elements of $R$ form a monoid isomorphic to $\mathbb{T}_{\mathbb{M}}$. For example, 
$\mathbb{R}$ could be a left Ore
Bezout domain. Then $\en(\mathbb{M})$ (which is isomorphic to $M_n(\mathbb{R})$) is a left order (in the sense of either ring or semigroup theory), in $\en(\A)$. The construction of
\cite{gouldI} in this case yields that
$\A$ is the $n$-dimensional vector space $\mathbb{D}\otimes \mathbb{M}$, where 
 $\mathbb{D}$ is the division ring of left 
quotients of $\mathbb{R})$, so that  $\en(\mathbb{A})$ is  isomorphic to the monoid $M_n(\mathbb{D})$.  Our results now yield (for $n\geq 2$) that $M_n(\mathbb{R})$ is straight in
$M_n(\mathbb{D})$ if and only if $M_n(\mathbb{R})$ is a two-sided order in $M_n(\mathbb{D})$ if and only if $\mathbb{R}$ is also right Ore if and only if $\mathbb{R}$ is also a right order in $\mathbb{D}$.
\end{example}

\begin{example} As explained in \cite[Section 2]{basisii}, any free left $\mathbb{T}$-act $\B$ of rank $n\in\mathbb{N}$ over a cancellative monoid $\mathbb{T}$ such that every finitely generated left ideal is principal is a stable basis algebra. Certainly $\mathbb{T}$ is left Ore, and if 
$\mathbb{G}$ is its group of left quotients, then $\A$ is isomorphic to the free
left $\mathbb{G}$-act on $n$ generators. Our results show that (for $n\geq 2$), $\en(\B)$ is a straight left order in $\en(A)$ if and only if it is a two-sided order, if and only if $\mathbb{T}$ is also right Ore.
\end{example}

\section{Fully stratified straight left orders}\label{sec:FS}
%

In \cite[Theorem 6.2]{basisii}, it was claimed that for any stable basis algebra $\B$, $\End_f(\B)$ is a fully stratified straight left order in {\em some} regular semigroup. 
However, the proof of this result depends on the invalid Proposition II.2.6 in \cite{mvl}. In this section, we show that we can still obtain that $\End_f(\B)$ a fully stratified straight  left order under certain conditions, closely related to those in Section~\ref{sec:right}. 
{\color{black} We remark first that if $\B$ has finite rank and satisfies the distributivity condition, then $\End_f(\B)=
\End(\B)$ and from Theorems~\ref{rightord} and ~\ref{thm:straight}, $\End(\B)$ is a fully stratified straight left order in $\End(\A)$ (where $\A$ is constructed as in Section~\ref{sec:right}) if and only if $\B$ satisfies  (CI) and if rank $\B \ge 2$, then the monoid
$\mathbb{T}$ is both left and right reversible. In the general case of arbitrary rank, where we have no construction of $\A$ to hand, we assume rather stronger conditions on $\B$ in order to find sufficient conditions for
$\End_f(\B)$ to be a fully stratified straight left order. Our approach is to check the list of conditions   for a semigroup to be a fully stratified straight left order given in  \cite[Proposition 3.2]{gould ab}. In fact, we are left with just two conditions to check, and for the first, we need make no additional assumptions.}

\begin{lemma} \label{l:eiil}
Let $\B$ be a stable basis algebra 
 and  let $\alpha,\beta \in \End_f(\B)$ be such that $\alpha \le_{\mathcal{L}^*}\beta$. Then there exist $\gam \in \End_f(\B)$, such that $\alpha \, \mathcal{L}^*\, \gam\beta $.
\end{lemma}
\begin{proof} Let  $B$ have basis $\{ b_i:i\in I\}$ so that  $B=\langle b_i:i\in I\rangle_{\mathbb{B}}$, and for any $i\in I$,  $b_i\not\prec \langle b_j:j\in I\setminus\{ i\}\rangle$. 

By Proposition~\ref{prop:greens}, $\alpha \le_{\mathcal{L}^*}\beta$ if and only if $\PC(\im \alpha) \subseteq \PC(\im \beta)$. Let $\{ x_1,\dots, x_m\}$ be a basis for $\PC(\im \alpha)$.  
First, if $m=0$, then $\im \alpha=\langle \emptyset \rangle_{\mathbb{B}}$, so that
$\alpha=\alpha\beta$ and so the lemma holds with $\gamma=\alpha$. We suppose therefore
that $m\in\mathbb{N}$.

Set $v_i=b_i\beta$, $i\in I$ so that $\im \beta =\langle v_i:i\in I\rangle$. 
There exists a $p\in\mathbb{N}$ and, for $j=1,\dots, m$,  $p$-ary term functions $t_j$ and $u_j \in T$ with 
$u_j(x_j)=t_j(v_1,\dots, v_p)$ (where we allow ourselves the freedom to relabel the
$b_i$ and hence $v_i$, as convenient).

Define $\gam \in \End{\B}$ by 
$ b_j\gam= t_j(b_1, \dots,b_p)$ for $j=1,\dots, m$, and $b_j\gam=b_1\gam$ else. 
Then $\gamma\in \End_f(\B)$ and
\begin{eqnarray*}
b_j \gam \beta&=&t_j(b_1, \dots, b_p)\beta \\
&=& t_j(b_1 \beta, \dots, b_p \beta) \\
&=& t_j(v_1,\dots, v_p) \\
&=& u_j(x_j)
\end{eqnarray*}
for $1\le j\le m$.
Moreover, for $j\notin\{ 1,\hdots, m\}$ we have $b_j \gam\beta =b_1 \gam \beta =u_1(x_1)$. We deduce
\begin{equation} \label{eq:PC}\im \gam \beta= \langle u_1(x_1), \dots, u_m(x_m) \rangle.\end{equation}
 so that clearly $\PC(\im \gam\beta) \subseteq  \PC(\im \alpha ).$
 Further, as  $x_i \in \PC(\im \gam \beta)$ for 
$1\le i \le m$, we have
$$\PC(\im \alpha)=\langle x_1, \dots, x_m \rangle \subseteq  \PC(\im \gam\beta ).$$
Hence $\PC(\im \alpha)=\pc(\im\gam\beta)$ so that from Proposition~\ref{prop:greens}, we have $\alpha\, \mathcal{L}^*\, \gam\beta $, as required.
\end{proof}

\begin{lemma} \label{l:eiir}
Let $\B$ be a stable basis algebra that satisfies the distributivity condition
and has no constants. Assume that $\T$ is commutative.
Let $\alpha,\beta \in \End_f(\B)$, such that $\alpha \le_{\R^*}\beta$. Then there exist $\gam \in \End_f(\B)$, such that $\alpha\, \R^*\, \beta \gam$.
\end{lemma}
\begin{proof} As in Lemma~\ref{l:eiil}, let  $B$ have basis $\{ b_i:i\in I\}$. 
Let $\{ u_1, \dots, u_m\}$ be a basis for $\pc(\im\beta)$. 
Set $c_i=b_i\beta$, so that 
$\im \beta= \langle c_i:i\in I\rangle$, and choose $m$-ary term functions $t_i$ such that $c_i=t_i(u_1,\dots ,u_m)$ for $i\in I$. As the $u_i$ are in $\PC(\im \beta)$, we may find $p_i \in T$ such that $p_i(u_i)\in
\im \beta$ for $1\le i \le m$. Let
$p=p_1\hdots p_m$ so that, as $\T$ is commutative,  $p(u_i) \in
\im \beta$ for $1\le i \le m$.
Since $m$ is finite, we may find $n\in\mathbb{N}$ and  $n$-ary term functions $s_i$, such that $p(u_i) = s_i(c_1,\dots, c_n)$ for $1\le i \le m$ (with, as earlier, some relabelling).

We  put
$\vec b=(b_1,\dots,b_n)$ and $\vec u=(u_1,\hdots, u_m)$. 
Extend $\{u_1,\dots, u_m\}$ to a basis  $\{ u_i:i\in I\}$ of $\B$, and define $\gam \in \End(\B)$ by
$$u_j \gam=s_j(b_1 \alpha, \dots, b_n \alpha)=s_j(\vec b \alpha)$$ for $1\le j\le m$, 
  and $u_j\gamma=u_1\gamma$ where $j\notin\{ 1,\hdots,m\}$.
  Clearly $\gamma\in \End_f(\B)$.


Let $w$ be an arbitrary $k$-ary term function,
where without loss of generality we assume $k\geq n$.
Let  $x=w(\bar b)$ where $\bar b=(b_1,\hdots ,b_k)$.  
Then 
\begin{eqnarray*} x\beta\gam&=&w(\bar b)\beta \gam\\
&=& w(b_1 \beta, \dots ,b_k \beta)\gam\\
&=& w(c_1, \dots ,c_k)\gam\\
&=& w(t_1(\vec u),\dots, t_k(\vec u))\gam\\
&=& w(t_1(s_1(\vec b \alpha), \dots, s_m(\vec b \alpha)),\dots,t_k(s_1(\vec b \alpha), \dots, s_m(\vec b \alpha)))\\
&=&\left(w(t_1(s_1,\dots,s_m), \dots, t_k(s_1,\dots, s_m))(\vec b)\right)\alpha\\
&=&\left(w(t_1(s_1,\dots,s_m), \dots, t_k(s_1,\dots, s_m))(\bar b)\right)\alpha
\end{eqnarray*}
where we reinterpret $(w(t_1(s_1,\dots,s_m), \dots, t_k(s_1,\dots, s_m))$ as being
$k$-ary, in the final step.
 Moreover, 
\begin{eqnarray*}
\left(pw(\bar b)\right)\beta&=&pw(b_1 \beta, \dots, b_k\beta)\\
&=& pw(c_1,\hdots, c_k)\\
&=&pw(t_1(\vec u), \dots, t_k(\vec u))\\
&=&  w(t_1(p(u_1),\dots p(u_m)), \dots, t_k(p(u_1 ),\dots, p(u_m)))\\
&=& \left(w(t_1(s_1,\dots, s_m),\dots, t_k(s_1,\dots, s_m))\right)(c_1,\hdots, c_n)\\ 
&=& \left(\left(w(t_1(s_1,\dots, s_m),\dots, t_k(s_1,\dots, s_m))\right)(\vec b)\right)\beta\\
&=& \left(\left(w(t_1(s_1,\dots, s_m),\dots, t_k(s_1,\dots, s_m))\right)(\bar b)\right)\beta.
\end{eqnarray*}
{\color{black} Here the fourth equation follows inductively from the distributivity condition, the fact that $\T$ is commutative, and that $\B$ has no constants.}

This means that 
$$\left(pw(\bar b), \left(w(t_1(s_1,\dots, s_m),\dots, t_n(s_1,\dots, s_m))\right)(\bar b)\right) \in \ker \beta$$ and so 

$$\left(pw(\bar b), \left(w(t_1(s_1,\dots, s_m),\dots, t_n(s_1,\dots, s_m))\right)(\bar b)\right) \in \ker \alpha,$$ 
as $\alpha \le_{\R^*}\beta$. 

Let $v,v'$ be arbitrary $k$-ary term functions. Applying the above to $y=v(\bar b)$ and $y'=v'(\bar b)$, and using the fact that $p$ is one-one, we get that
\begin{eqnarray*}&&
y\beta \gam= y'\beta \gam \\
\Leftrightarrow & &\left( \left(v(t_1(s_1,\dots, s_m),\dots, t_n(s_1,\dots, s_m))\right)(\bar b)\right) \alpha\\
&=&\left(\left(v'(t_1(s_1,\dots, s_m),\dots, t_n(s_1,\dots, s_m))\right)(\bar b)\right)\alpha\\
\Leftrightarrow && \left((pv)(\bar b),(pv')(\bar b)\right)\in \ker \alpha\\
\Leftrightarrow&&pv(b_1 \alpha, \dots, b_k\alpha)=pv'(b_1 \alpha, \dots, b_k\alpha)\\
\Leftrightarrow&&v(b_1 \alpha, \dots, b_k\alpha)=v'(b_1 \alpha, \dots, b_k\alpha)\\
\Leftrightarrow && (v(\bar b))\alpha=(v'(\bar b))\alpha\\
\Leftrightarrow && y\alpha =y'\alpha
\end{eqnarray*}
It follows that $\alpha\, \R^*\, \beta \gam$.
\end{proof}

{\color{black} Before stating and proving the main result of this section, establishing {\bf Result 3}, we remind the reader that the relations $\ars$ and $\els$ are the equivalence relations associated, respectively, with the pre-orders
$\leqars$ and $\leqels$ of Proposition~\ref{prop:greens}.}

\begin{theorem}\label{thm:end_f} Let $\B$ be a stable basis algebra  that satisfies the distributivity condition and has no constants. Assume that $\T$ is commutative. Then $\End_f(\B)$ is a fully stratified straight left order in a regular semigroup.
\end{theorem}
\begin{proof} From \cite[Proposition 3.2]{gould ab} the semigroup $\End_f(\B)$is a fully stratified straight left order precisely when the following conditions, together with
the left-right duals (Eii)(r),(Eiii)(r), (Evi)(r),  of (Eii)(l), (Eiii)(l), (Evi)(l), respectively, are satisfied:
\begin{itemize}
\item[(Ei)] $\mathcal{L}^*\circ \mathcal{R}^*=\mathcal{R}^*\circ\mathcal{L}^*$.
\item[(Eii)(l)] For all $\alpha,\beta \in \End_f(\B)$, $\alpha\le_{\mathcal{L}^*} \beta$  if and only if $\alpha\, \mathcal{L}^*\,  \gamma \beta$ for some $\gamma \in \End_f(\B)$. 
\item[(Eiii)(l)] Every $\mathcal{L}^*$-class contains a square-cancellable endomorphism.
\item[(Evi)(l)] For all square-cancellable $\alpha \in \End_f(\B)$, and all $\beta, \gamma \in \End_f(\B)$, if $\beta, \gamma \le_{\mathcal{L}^*}\alpha$ and 
$\beta \alpha=\gamma\alpha$, then $\beta =\gamma$.
\item[(Evii)(r)] For all square-cancellable $\alpha \in \End_f(\B)$, and all $\beta, \gamma \in \End_f(\B)$, if $\beta, \gamma \le_{\mathcal{R}^*}\alpha$ and
$\alpha \beta \mathcal{R}^* \alpha \gamma$, then $\beta \mathcal{R}^* \gamma$.
\item[(Gii)] If $\alpha \in \End_f(\B)$ is square-cancellable, then $H^*_\alpha$ is left Ore. 
\end{itemize}

From \cite[Corollary 6.3]{basisii}, $\End_f(\B)$ is abundant. 
As pointed out at the beginning of \cite[Section 4]{gould ab}, (Eiii)(l) and (r)  hold in any abundant semigroup. In \cite{basisii}, it is shown 
that (Ei), (Evi)(l), (Evi)(r), (Evii)(r), and (Gii) hold in $\End_f(\B)$, namely in Corollary 6.3, Lemma 6.9, Lemma 6.10, Lemma 6.11, and Corollary 6.6, respectively.
 These results do not depend on  the incorrect \cite[Proposition II.2.6]{mvl}, although  note that another statement in Corollary 6.3 does.  

It remains to establish the two version of (Eii). It follow from Proposition~\ref{prop:greens} that if $\alpha\,\ars\, \beta\gamma$ for any $\alpha, \beta,\gamma\in \End_f(\B)$, then $\ker\beta\subseteq \ker \alpha$, so that
$\alpha\leq_{\mathcal{R}^*} \beta$,  with a dual statement for
$\els$. We then use  Lemma \ref{l:eiil} for condition (Eii)(l) and in Lemma \ref{l:eiir} for Condition (Eii)(r).  The result follows.
\end{proof}

{\color{black}
\section{Open questions}\label{sec:open}

The main aim of this article was to show that not all stable basis algebras satisfy the distributivity condition, achieved in Section~\ref{sec:example}, and to complete the investigation of the left order $\End(\B)$ in $\End(\A)$ constructed in \cite{gould}.
However, achieving our objectives here, and  the discovery of the knock-on effects of the error in \cite{mvl}, has prompted a number of further questions which we would like to pose.
{\color{black}
\begin{question}\label{qn:1} Is every stable basis algebra $\B$ a reduct of an independence algebra? 
\end{question}

{\color{black} It would make sense to consider first the special case  where $\B$ has finite rank.  Question~\ref{qn:1} is closed tied to: }

\begin{question} If a stable basis algebra $\B$ is a reduct of an independence algebra $\A$, then does
$\B$ satisfy the distributivity condition if and only if $\A$ does?
\end{question}

It may be that in the above, one would want only to consider the case where $\B$ contains all the non-unary basic operations of $\A$.

\begin{question} Let $\B$ be a stable basis algebra. Is $\End_f(\B)$ is a left order in
a regular semigroup if and only if $\B$ is a reduct of an independence algebra
$\A$ such that $\End_f(\B)$ is a left order in $\End_f(\A)$?
\end{question}

Finally, following from Theorem~\ref{thm:end_f}  we ask:

\begin{question}  Let $\B$ be a stable basis algebra satisfying the distributivity condition. Is  $\End_f(\B)$  a 
fully stratified straight left order if and only if $\B$ satisfies (CI) and (if rank $\B \ge 2$) the monoid
$\mathbb{T}$ is both left and right reversible? In this case, is
$\B$ a reduct of an independence algebra $\A$ such that $\End_f(\B)$ is a 
fully stratified straight left order in $\End_f(\A)$?
\end{question}
}

\end{document}